\documentclass{amsart}

\usepackage[utf8]{inputenc}
\usepackage{latexsym,amssymb}
\usepackage{hyperref}
\usepackage{diagrams}
\newarrow{To}----{->}

\newtheorem{theorem}{Theorem}
\newtheorem{lemma}[theorem]{Lemma}
\newtheorem{proposition}[theorem]{Proposition}

\newtheorem{definition}[theorem]{Definition}
\newtheorem{remark}[theorem]{Remark}

\renewcommand{\SS}{\mathbb{S}}

\newcommand\mES[1]{\maE_{\maS}(#1)}
\newcommand\maS{\mathcal{S}}

\newcommand\maE{\mathcal{E}}

\begin{document}

\title[Essential Spectrum]
{exhaustive families of representations of $C^*$-algebras 
associated to $N$-body Hamiltonians with asymptotically
homogeneous interactions}

\author[J. Mougel]{J\'er\'emy Mougel} \address{Universit\'{e} de
  Lorraine, UFR MIM, 3 rue Augustin Fresnel 57045 METZ, France}
\email{jeremy.mougel@univ-lorraine.fr}

\author[N. Prudhon]{Nicolas Prudhon} \address{Universit\'{e} de
  Lorraine, UFR MIM, 3 rue Augustin Fresnel 57045 METZ, France}
\email{nicolas.prudhon@univ-lorraine.fr}

\begin{abstract}
We continue the analysis of algebras 
introduced by Georgescu,  Nistor and their coauthors, 
in order to study $N$-body type Hamiltonians with interactions.
More precisely, let $Y \subset X$ be a linear subspace of 
a finite dimensional Euclidean space $X$,
and $v_Y$ be a continuous function on $X/Y$ that has uniform
homogeneous radial limits at infinity. We consider, in this paper,
Hamiltonians of the form $H = - \Delta + \sum_{Y \in \maS} v_Y$, where
the subspaces $Y \subset X$ belong to some given family $\maS$
of subspaces. Georgescu and Nistor have considered the case when
$\maS$ consists of {\em all} subspaces $Y \subset X$, and 
Nistor and the authors considered the case when $\maS$ is a finite semi lattice
and Georgescu generalized these results to any families. 
In this paper, we develop new techniques to prove their results 
on the spectral theory of the Hamiltonian to the case where $\maS$ is any
family of subspaces also, and extend those results to other operators 
affiliated to a
larger algebra of pseudo-differential operators associated to the action of $X$
introduced by Connes. 
In addition, we exhibit Fredholm conditions for such elliptic operators. 
We also note that the algebras we consider 
answer a question of Melrose and Singer. 
\end{abstract}

\maketitle 



An new approach in the study of Hamiltonians of $N$-body type with
interactions that are asymptotically homogeneous at infinity on a
finite dimensional Euclidean space $X$ was initiated by Georgescu 
and Nistor \cite{Georgescu, GN, Georgescu18}. 

For any finite real vector space $Z$, we let $\overline{Z}$ denote its
spherical compactification. A function in $C(\overline{Z})$
is thus a continuous function on $Z$ that has uniform radial limits at
infinity. Let $\SS_Z$ be the set of half-lines in $Z$,
that is $\SS_Z := \{\, \hat{a},\ a\in Z, a \neq 0 \,\}$ where $\hat{a}
:= \{ra, \, r>0\}$. We identify $\SS_Z = \overline{Z}\smallsetminus Z$ . 

For any subspace $Y \subset X$, $\pi_{Y} : X \to X/Y$
denotes the canonical projection.  Let
\begin{equation}\label{eq.simple.H}
 H = - \Delta + \sum_{Y \in \maS} v_Y \ , 
\end{equation}
where $v_Y \in C(\overline{X/Y})$ is seen as a bounded continuous 
function on $X$ via the projection $\pi_Y : X \to X/Y$. 
The sum is over {\em all}
subspaces $Y \subset X$, $Y \in \maS$ and is assumed to be
uniformly convergent. One of the main results of \cite{GN, MNP}
describe the essential spectrum of $H$ extending the celebrated HVZ theorem \cite{Simon4}.  
The goal of this paper
is to explain how these results can be extended to any family of 
subspaces that contains $\{0\}$ and to more general operators using $C^*$-algebras techniques. 

Let $\maS$ be a family of subspaces of $X$ with $0\in \maS$. 
We define the commutative sub-$C^*$-algebra $\maE_\maS(X)$
of the commutative $C^*$-algebra $C_b^u(X)$ of bounded uniformly continous functions on $X$ by
\begin{equation}\label{eq.def.maEX}
 \maE_\maS(X) = \langle C(\overline{X/Y}) \ , \quad  Y \in \maS\ \rangle \subset  C_b^u(X).
\end{equation}
The algebras $\maE_\maS(X)$ give an answer to a question of 
Melrose and Singer \cite{MelroseSinger}.
\begin{theorem} Let $n$ be an integer.
Let $\maS^n$ be the semi-lattice of subspaces of $X^n$ generated by
$\maS_{i}^n \cup \maS_{ij}^n$ where
\begin{align*}
\maS_i^n=\{(x_1,\ldots,x_n)\in X^n\,;\, x_i=0 \}\,\\
\maS^n_{ij}=\{(x_1,\ldots,x_n)\in X^n\,;\, x_i=x_j \}
\end{align*}
Then the spectrum $\Omega_{\maS^n}$ of $\maE_{\maS^n}(X^n)$ is a 
compactification of $X^n$ satisfying the following properties :
\begin{enumerate}
\item $\Omega_{\maS^1}$ is the spherical compactification $\overline{X}$,
\item The action of the symmetric group $\mathfrak{S}_n$ on $X^n$ extends
continuously to $\Omega_{\maS_n}$,
\item The projections $p^{n,k}_I\colon X^n \to X^k$, 
$p^{n,k}_I(x_1,\ldots,x_n)=(x_{i_1},\ldots,x_{i_k})$ extend continuously
to $p^{n,k}_I\colon \Omega_{\maS_n} \to \Omega_{\maS_k}$,
\item The difference maps $\delta_{ij}(x_1,\ldots, x_n)=x_i-x_j$ from
$X^n$ to $X$ extend continuously to the compactifications.
\end{enumerate}
\end{theorem}
Actually, the spectrum $\Omega_{\mathcal{S}^n}$ have very strong connection with the space built by Vasy in \cite{VasyReg}
and generalized by Kottke in the last section of \cite{Kottke}. \\
The additive group $X$ acts by translation on $C_b^u(X)$ 
and the subalgebra $\maE_\maS(X)$ is invariant. 
So a crossed product $C^{*}$-algebra is obtained
\begin{equation}\label{eq.def.rX}
 \maE_\maS(X)\rtimes X \ ,
\end{equation}
which can be regarded as an algebra of operators on $L^2(X)$. 
Thanks to the assumption $0\in\maS$, the algebra $C_0(X)$ belongs
$\maE_\maS(X)$. Hence $C_0(X)\rtimes X$ is contained in 
$\maE_\maS(X)\rtimes X$. It follows from the definition of crossed products
algebras that the $C^*$-algebra $\mES X \rtimes X$ is generated by two kinds of operators : 
multiplication operators $m_f$ associated to functions $f\in \mES X$, and 
convolution operators
$$C_\phi u(x) := \int_{X} \phi(y) u(x-y)dy$$ 
with $\phi \in C_c(X)$, a continuous compactly supported function. 
An immediate computation shows that $m_fc_\phi$ (resp. $c_\phi m_f$) 
is a kernel operator with kernel
\begin{equation} \label{mfcphi-eq}
K(x,y)=f(x)\phi(y-x)\,,\qquad (\text{resp. } K(x,y)=f(y)\phi(y-x)).
\end{equation}
\begin{proposition} \label{mfcphi}
(i) The subalgebra $C_0(X)\rtimes X$ is the algebra $\mathcal{K}(X)$
of compact operators on $L^2(X)$. \\
(ii) For $f\in C(\overline{X})$ and $\phi \in C_c(X)$ the commutator
$[m_f,c_\phi]$ is compact.
\end{proposition}
The point $(i)$ is a consequence of equation \eqref{mfcphi-eq}
because the kernel $K$ has compact support when $f$ does and the 
result follows by density. 
Again, thanks to equation \eqref{mfcphi-eq}, one sees that
the commutator is a kernel operator with kernel
$$ K(x,y)=\phi(y-x)(f(x)-f(y))\,.$$
Hence, in view of $\phi \in C_c(X)$, the support of $K$ is contained in
a band around the diagonal. The distance between the border of the band and the diagonal is bounded.
Moreover, $K$ goes to $0$ at infinity because $f$ has 
radial limits. So the commutator is a limit of Hilbert-Schmidt operators, and hence is compact.

Recall that a self-adjoint operator $P$ on $L^2(X)$ is said to be affiliated to a
$C^*$-algrebra $A$ of bounded operators on $L^2(X)$ if
for some (and hence any) function $h\in C_0(\mathbb{R})$ then
$h(P)$ belongs to $A$. For example,, it follows from the identity
$$ 
(H+i)^{-1}=(-\Delta +i)^{-1}\left(1+V(-\Delta+i)^{-1}\right)^{-1}\,,
$$
that $H$ is affiliated to $\maE_\maS(X)\rtimes X$.
More generally, for any $C^*$-algebra $A$, a morphism $h\colon C_0(\mathbb{R}) \to A$ is called
an operator affiliated to $A$. 
Following Connes \cite{connes} and Baaj \cite{baaj} we introduce the $C^*$-algebra of 
non positive order pseudo-differential operators $\Psi\mathrm{DO}(\maE_\maS(X),X)$
together with the symbol map exact sequence
$$
0 \to \maE_\maS(X)\rtimes X \to 
\Psi\mathrm{DO}(\maE_\maS(X),X) 
\mathop{\longrightarrow}\limits^{\sigma_0}
C(\SS_X, \maE_\maS(X)) \to 0 \,.
$$
Positive order pseudo-differential operators
are examples of operators affiliated to the algebra of non positive order
pseudo-differential operators $\Psi\mathrm{DO}(\maE_\maS(X),X)$.

Let $\alpha \in \mathbb{S}_X$. 
For each $x\in X$, we let $(T_x f)(y)=f(y-x)$ denote the
translation on $L^2(X)$. For any operator $P$ on $L^2(X)$, we let
\begin{equation}\label{eq.def.tau}
 \tau_{\alpha}(P) = \lim_{r\to+\infty} T_{ra}^* P T_{ra}\ , \quad
 \mbox{if } \ \alpha = \hat{a} \in \SS_X\ ,
\end{equation}
whenever the {\em strong} limit exists. 
\begin{lemma} \label{tau}
For $f \in C(\overline{X/Y})$ one has
$$
\tau_\alpha(f)(x)=
\left\{\begin{array}{ll}
f(x) & \text{ if } Y\supset \alpha \,,\\
f(\pi_Y(\alpha)) & \text{ else. }
\end{array}\right.
$$
\end{lemma}
We define $\mathcal{S}_\alpha=\{Y\in\maS\,;\, \alpha \subset Y\}$.
It follows from the previous lemma that on $\maE_\maS(X)$, $\tau_\alpha$
is the projection on the subalgebra $\maE_{\maS_\alpha}(X)$,
$$
\tau_\alpha \colon \maE_\maS(X) \to \maE_{\maS_\alpha}(X)\,.
$$

\begin{theorem}\label{thm.main} 
\begin{enumerate}
\item Let $P$ be a self-adjoint operator affiliated to 
$\Psi\mathrm{DO}(\maE_\maS(X),X)$ and $\alpha = \hat{a} \in\SS_X$. Then the
limit $\tau_{\alpha}(P) := \lim_{r\to+\infty} T_{ra}^* P T_{ra}$
exists and
  \begin{equation*}
   \mathrm{Spec}_\mathrm{ess}(P) = \cup_{\alpha\in\SS_X} \
   \mathrm{Spec}(\tau_{\alpha}(P))\ .
  \end{equation*} 
\item  Let $P\in \Psi\mathrm{DO}(\maE_\maS(X),X)$. Then $P$ is a Fredholm operator if and only if
$P$ is elliptic (i.e. $\sigma_0(P)$ is invertible) and for all $\alpha\in\SS_X$, $\tau_\alpha(P)$ is invertible.
\end{enumerate}
\end{theorem}

This extends theorems of \cite{GN, MNP} in the following sense :
only operators affiliated to $\maE_\maS(X) \rtimes X$
are considered there, and the relation is
\begin{equation}\label{GN}
 \mathrm{Spec}_\mathrm{ess}(H) = \overline{\cup}_{\alpha\in\SS_X}
 \mathrm{Spec}(\tau_{\alpha}(H))
\end{equation}
in \cite{GN}. In \cite{MNP} only finite semi-lattice $\maS$ are considered.
The equation \eqref{GN} means that the family
$(\tau_\alpha)$ is a {\em faithful} family of morphism of
$\maE_\maS(X)\rtimes X$. The stronger result of \cite{MNP} is
obtained by showing that the family $(\tau_\alpha\rtimes X)_{\alpha\in \SS_X}$ is actually an
{\em exhaustive} family of representations of $\maE_\maS(X)\rtimes X$, 
when $\maS$ is a finite semi-lattice.
In the framework of admissible locally compact group, decomposition of essential
spectrum involving {\em exhaustive} families can be found in \cite{Mantoiu2} \cite{Mantoiu1}.
In fact, by \cite[Proposition 3.12]{NP}, exhaustive families are also
strictly spectral families in the following sense.

\begin{definition} \cite {NP,Roch} 
\begin{enumerate}
\item A family $(\phi_i)_{i\in I}$ of morphisms of a 
$C^*$-algebra $A$ is said to be \emph{exhaustive} if any primitive
ideal contains at least $\ker \phi_i$ for some $i\in I$.
\item A family  $(\phi_i)_{i\in I}$ of morphisms of a unital
$C^*$-algebra $A$ is said to be \emph{strictly spectral} if 
$$
(\forall a\in A)\qquad \mathrm{Spec}(a)=\cup_{i\in I}\mathrm{Spec}(\phi_i(a))
$$
\end{enumerate}
\end{definition}

\begin{theorem}
Let $\maS$ be a family of subspaces of $X$ with $0\in \maS$. 
Then the family $(\tau_\alpha\rtimes X)_{\alpha\in \SS_X}$ 
is an exhaustive family of $\maE_\maS(X)\rtimes X/\mathcal{K}(X)$. 
\end{theorem}
Let us prove this result.
Let $\pi$ be an 
irreducible representation of $\maE_\maS(X)\rtimes X/\mathcal{K}(X)$.
It extends to an irreducible representation of $\maE_\maS(X)\rtimes X$
as well as to their multipliers algebras $\mathcal{M}(\maE_\maS(X)\rtimes X/\mathcal{K}(X))$
and $\mathcal{M}(\maE_\maS(X)\rtimes X)$. By proposition \ref{mfcphi}(i),
one obtains the following commutative diagram:
\begin{equation}\label{main_diagram}
\begin{diagram} 
 C(\overline{X})
  & \hookrightarrow
  & \maE_\maS(X)
  & \rTo
  & \mathcal{M}(\maE_\maS(X)\rtimes X)
  & 
  & \\
 \dTo^{p} 
  & & &
  & \dTo
  & 
  & \\
 C(\SS_X)
  &  & \rTo^\phi &
  & \mathcal{M}(\maE_\maS(X)\rtimes X/\mathcal{K}(X))
  & \rTo^\pi
  & \mathcal{B}(\mathcal{H}_\pi)
\end{diagram}
\end{equation}
\begin{lemma}
The image $\phi(C(\SS_X))$ is central in
$\mathcal{M}(\maE_\maS(X)\rtimes X/\mathcal{K}(X))$.
\end{lemma}
In fact it is enough to show that any $f\in C(\overline{X})$
commutes with any element of $\maE_\maS(X)\rtimes X$ modulo a compact
operator. But the result is true on the generators by 
Proposition \ref{mfcphi}(ii), so  the lemma follows by density.
 
By the Schur Lemma, we deduce that $\pi\circ \phi$ is a character
of $C(\SS_X)$. Hence there exists some $\alpha\in\SS_X$ such that
$\pi|_{C(\overline{X})}=\chi_\alpha I$, where $\chi_\alpha$ is
the character of $C(\overline{X})$ given by the evaluation at 
$\alpha \in\SS_X$.
\begin{proposition}
One has $\ker \tau_\alpha=(\ker \chi_\alpha)\maE_\maS(X)$.
\end{proposition}
\begin{proof}
We need to show that $\maE_\maS(X)/\ker \tau_\alpha=\maE_{\maS_\alpha}(X)$ and 
$\maE_\maS(X)/(\ker \chi_\alpha)\maE_\maS(X)$ have the same characters.
By definition, for any character $\chi$ of $\maE_{\maS_\alpha}(X)$,
there exists a unique character 
$\chi'$ of $\maE_{\maS}(X)$ such that $\chi'=\chi\circ\tau_\alpha$.
In view of lemma \ref{tau}, this is equivalent to the following :
\begin{equation}\label{taueq}
(\forall Y\in \maS, \alpha\not\subset Y, \forall u\in C(\overline{X/Y})) \quad \chi(u)=u(\pi_Y(\alpha)).
\end{equation}
In particular, for $Y=0$, we see that $\chi_{|C(\overline{X})}=\chi_\alpha$.
Reciprocally it follows from \cite[Lemma 6.7]{GN} that if 
$\chi_{|C(\overline{X})}=\chi_\alpha$ then relation \eqref{taueq} is true.
On the other hand, the characters of 
$\maE_\maS(X)/(\ker \chi_\alpha)\maE_\maS(X)$
are precisely the characters $\chi$ of $\maE_\maS(X)$ such that
$\chi_{C(\overline{X})}=\chi_\alpha$. So $\ker \tau_\alpha=(\ker\chi_\alpha)\maE_\maS(X)$ as claimed. 
\end{proof}
Now if $\pi_{|C(\overline{X})}=\chi_\alpha$, one has 
$\ker \pi \supset (\ker \chi_\alpha)\maE_\maS(X)=\ker \tau_\alpha$.
Finally, 
$$
\ker (\tau_\alpha\rtimes X) =
(\ker \tau_\alpha)\rtimes X \subset
\ker \pi \,.
$$
It follows that $(\tau_\alpha\rtimes X)_{\alpha \in \SS_X}$ is 
an exhaustive family of morphisms.
\begin{remark}
The results presented here can easily be extended to pseudo-differential operators with matrix coefficients.
For example, Dirac operators
$ D_V = D+V\,,$
 with potentials $V$ as in \eqref{eq.simple.H} may be considered and satisfy the condition of Theorem
\ref{thm.main}.
\end{remark}
See also \cite[Example 6.35]{GI3} for others physical interesting operators.

\subsection*{Acknowledgments} 
The authors thank Victor Nistor  for useful discussions. 

\bibliographystyle{plain} 
\bibliography{MP18-cras}

\end{document}